\documentclass{amsart}
\usepackage[english]{babel}
\usepackage[utf8]{inputenc}
\usepackage{amsmath, amssymb,epic,graphicx,mathrsfs,enumerate}
\usepackage[all]{xy}
\usepackage{color}
\usepackage{comment}
\usepackage{enumitem}
\usepackage{hyperref}
\usepackage[]{frontespizio}

\usepackage{amsmath}
\usepackage{amsthm}
\usepackage{amssymb}
\usepackage{latexsym}
\usepackage{epsfig}

\DeclareMathOperator{\frat}{Frat} 
\DeclareMathOperator{\aut}{Aut} 
\DeclareMathOperator{\soc}{soc}

\DeclareMathOperator{\diam}{diam}
\DeclareMathOperator{\kndo}{End}
 
\DeclareMathOperator{\syl}{Syl}
\DeclareMathOperator{\B}{B}
\DeclareMathOperator{\I}{I}

\newcommand{\sym}{\mathrm{Sym}}
\newcommand{\supp}{\mathrm{supp}}
\newcommand{\Aut}{\mathrm{Aut}}
\newcommand{\psl}{\mathrm{PSL}}

\newcommand{\gen}[1]{\left\langle#1\right\rangle}

\newcommand{\ifa}{if and only if }

\DeclareMathOperator{\ssl}{SL}

\newtheorem{thm}{Theorem}
\newtheorem{cor}[thm]{Corollary}
 \newtheorem{lemma}[thm]{Lemma}
\newtheorem{prop}[thm]{Proposition} 
 \newtheorem{defn}[thm]{Definition}

\numberwithin{equation}{section}

\renewcommand{\footnote}{\endnote}
\newcommand{\ignore}[1]{}\makeglossary

\begin{document}
	\bibliographystyle{amsplain}
\title[On the connectivity of the non-generating graph]{On the connectivity of the non-generating graph}

\author{Andrea Lucchini}
\address{Andrea Lucchini\\ Universit\`a di Padova\\  Dipartimento di Matematica \lq\lq Tullio Levi-Civita\rq\rq\\ Via Trieste 63, 35121 Padova, Italy\\email: lucchini@math.unipd.it}
\author{Daniele Nemmi}
\address{Daniele Nemmi\\ Universit\`a di Padova\\  Dipartimento di Matematica \lq\lq Tullio Levi-Civita\rq\rq\\ Via Trieste 63, 35121 Padova, Italy\\email: dnemmi@math.unipd.it}


\begin{abstract} Given a 2-generated finite  group $G$,  the non-generating graph   of $G$ has as  vertices the  elements of $G$ and two vertices  are adjacent if and only if they are distinct and do not generate $G$. We consider the graph $\Sigma(G)$ obtained from the non-generating graph of $G$ by deleting the universal vertices. We prove that if the derived subgroup of $G$ is not nilpotent, then this graph is connected, with diameter at most 5. Moreover we give a complete classification of the finite groups $G$ such that $\Sigma(G)$ is disconnected.	\end{abstract}

\maketitle

\hbox{}

\bibliographystyle{alpha}

\section{Introduction}

Let $G$ be a finite group. The generating graph for $G,$ written $\Gamma(G),$ is the graph in which the vertices are the elements of $G$ and there is an edge between $g_1$ and $g_2$ if $G$ is generated by $g_1$ and $g_2.$ If $G$ is not 2-generated, then there will be no edge in this graph. Thus, it is natural to assume that $G$ is 2-generated.
Quite a lot is known about this graph when $G$ is a non-abelian simple group; for example Guralnick and Kantor \cite{guka} showed that there is no
isolated vertex in $\Gamma(G)$ but the identity, and Breuer, Guralnick, Kantor
\cite{BGK} showed that the diameter of the subgraph of $\Gamma(G)$ induced by non-identity elements is  $2$
for all $G$.
If $G$ is an arbitrary finite groups, then $\Gamma(G)$ could contain many isolated vertices.
Let $\Delta(G)$ be the subgraph of $\Gamma(G)$ that is induced by all  the vertices that are not isolated.
In \cite{cl} and \cite{diam} it is  proved that if $G$ is a 2-generated soluble group, then $\Delta(G)$ is connected and $\diam(\Delta(G))\leq 3.$ 
The situation is  different if the solubility assumption is dropped. It is an open problem whether or not  $\Delta(G)$ is connected, but even when $\Delta(G)$ is connected, its diameter
can be arbitrarily large. For example if $G$ is the largest 2-generated direct power of $\ssl(2,2^p)$ and $p$ is a sufficiently large odd prime, then $\Delta(G)$ is connected but $\diam(\Delta(G))\geq 2^{p-2}-1$
(see \cite[Theorem 5.4]{CL2}).
	
\

The aim of this paper is to investigate the connectivity of the
complement graph, denoted by $\Sigma(G),$ of $\Delta(G).$ This graph can be described as follows: we take the  non-generating graph of $G,$ i.e. the graph whose vertices are the elements of $G$ and where there is an edge between $g_1$ and $g_2$ if  $\langle g_1, g_2\rangle\neq G$ and we remove the universal vertices (corresponding to the isolated vertices of the generating graph).
We prove that $\Sigma(G)$ is connected, except for some families that can be completely described. In any case, if $\Sigma(G)$ is disconnected, then $G$ is soluble and its derived subgroup is nilpotent.

\begin{thm}\label{main}
Let $G$ be a 2-generated finite group. Then $\Sigma(G)$ is connected if and only if none  of the following occurs:
\begin{enumerate}
	\item $G$ is cyclic;
	\item $G$ is a $p$-group;
	\item  $G/\frat(G)\cong (V_1\times \dots \times V_t)\rtimes H,$
	where $H\cong C_p$ for some prime $p,$ and $V_1,\dots,V_t$
	are pairwise non $H$-isomorphic non-trivial irreducible $H$-modules.
		\item  $G/\frat(G)\cong (V_1\times \dots \times V_t)\rtimes H,$
		where $H\cong C_p\times C_p$ for some prime $p,$  $V_1,\dots,V_t$
		are pairwise non $H$-isomorphic
		non-trivial irreducible $H$-modules and $C_H(V_1\times \dots \times V_t)\cong C_p.$
	\end{enumerate}
	Moreover if $\Sigma(G)$ is connected, then $\diam(\Sigma(G))\leq 5$, and
	$\diam(\Sigma(G))\leq 3$ under the additional assumption that $G$ is soluble.
\end{thm}
We don't know whether the bound $\Sigma(G)$ is the best possible. In any case if $\B$ is the Baby Monster, then $\diam(\Sigma(G))\geq 4$ (see the end of Section 	\ref{secunso}). On the other hand for soluble groups, the bound $\diam(\Sigma(G))\leq 3$ is the best possible. Consider for example $G=\langle a \rangle \times \langle b \rangle \times
\langle c \rangle,$ with $|a|=|b|=2$ and $|c|=3.$ Then the shortest path in $\Sigma(G)$ between $ac$ and $bc$ is
$(ac,a,b,bc).$

\

When $\Sigma(G)$ is disconnected, it is possible that it contains some isolated vertices. However this occurs only in few particular cases.
\begin{prop}\label{isola}Let $G$ be a 2-generated finite group. Then
	$\Sigma(G)$ has an isolated vertex \ifa 
	\begin{enumerate}
		\item $G$ is cyclic;
		\item $G\cong C_2\times C_2$;
		\item $G\cong D_p$ is a dihedral group with $2p$ elements, for $p$ an odd prime.
	\end{enumerate}
\end{prop}

The structure of the paper is as follows. In Section \ref{prsogr}
we study $\Sigma(G)$ in the particular case when $G$ is a primitive soluble group. In Section \ref{secunso} we investigate $\Sigma(G)$ when $G$ is a monolithic group whose socle is non-abelian. Thanks to the fact that if $N$ is a proper normal subgroup of $G$ and $\Sigma(G/N)$ is connected, then $\Sigma(G)$ is also connected
(see Proposition \ref{quoz}), using the results from Sections \ref{prsogr} and \ref{secunso} it can be easily proved that
the statement of Theorem \ref{main} holds if the derived subgroup of $G$ is not nilpotent. The case when the derivated subgroup is nilpotent is analysed in Section \ref{ultimasec}.

\section{Primitive soluble groups}\label{prsogr}

\begin{defn} Let $G$ be a finite group. We denote by $V(G)$ the subset of $G$ consisting of the elements $x$ with the property that $G=\langle x, y\rangle$ for some $y.$
\end{defn}

\begin{lemma}\cite[Proposition 2.2]{nof}\label{lemma10}
	Let $G$ be  a primitive  soluble group. Let $N=\soc(G)$ and $H$ a core-free maximal subgroup of $G$. Given $1\neq  h\in H$ and $n\in N$, $hn \in V(G)$ if and only if $h\in V(H).$ In particular $V(H)\setminus \{1\}\subseteq V(G).$
\end{lemma}

\begin{prop}\label{monosol}
Let $G$ be a 2-generated primitive soluble group. Then $\Sigma(G)$ is disconnected  if and only if either $G\cong C_p$ or $G/\soc(G)\cong C_p$. Moreover if $G$ is connected, then $\diam(\Sigma(G))\leq 3.$
\end{prop}
\begin{proof}
	If $G$ is nilpotent, then $G\cong C_p.$ In this case $\Sigma(G)$ is a disconnected graph with $p$ vertices and no edge. So we may assume that $G$ is not nilpotent. Thus $G\cong V \rtimes H,$ where $V=\soc(G)$ is a faithful irreducible $H$-module.  If $|H|=p,$ with $p$ a prime, then $\langle g, v \rangle=G$ for any $1\neq v \in V$ and $g\notin V.$ In particular $V\setminus \{1\}$ is a connected component of $\Sigma(G)$. So we may assume that $|H|$ is not a prime.
		Suppose that $g_1=v_1h_1, g_2=v_2h_2$ are two different elements of $V(G),$ with $v_1, v_2 \in V, h_1, h_2 \in H.$  By Lemma \ref{lemma10}, for $i\in \{1,2\},$ either $h_i=1$ (and therefore $H$ is cyclic) or $h_i \in V(G).$
	 If neither $h_1$ nor $h_2$ is a generator of $H$, then $(v_1h_1,h_1,h_2,v_2h_2)$ is a path in $\Sigma(G)$. So we may assume $H=\langle h_1 \rangle.$ In this case $h_1$ and $g_1$ are conjugated in $G$, so it is not restrictive to assume $v_1=1.$ 
	Since $|H|$ is not a prime, we can choose  $1\neq h \in H$ with $|h|<|H|.$ Since all the complements of $V$ in $G$ are conjugated, there exists $v\in V$ such that 
	$\langle g_2 \rangle \leq \langle h_1^v\rangle.$
But then $(h_1,h,h^v,g_2)$  is a path in $\Sigma(G).$ 
	\end{proof}

\section{Monolithic groups with non-abelian socle}\label{secunso}
Let $G$ be a 2-generated finite monolithic group, with a non-abelian socle. The aim of this section is to prove that the graph $\Sigma(G)$ is connected, with diameter at most 5.

Assume $A=\soc(G)\cong S^n$, with $S$ a finite non-abelian simple group and $n\in \mathbb N.$
We may identify $G$ with a subgroup of $\Aut(S^n)=\Aut S \wr \sym (n)$, the wreath product of $\Aut
S$ with the symmetric group of degree $n$. So the elements of $G$ are of the
kind $g=(\alpha_1, \dots, \alpha_n)\sigma$, with $\alpha_i \in \Aut S$ and $
\sigma \in \sym(n)$.
For all this section we will refer to this identification and we will denote by $\pi$ the homomorphism $\pi: \Aut S \wr \sym (n) \to \sym (n)$ mapping
$(\alpha_1, \dots, \alpha_n)\sigma$ to $\sigma$.

We begin with two lemmas concerning some properties of $\aut S.$
\begin{lemma}\label{accar}Let $S$ be a finite non-abelian simple group. There exist a subgroup $H$ of $\aut S$ and a prime divisor $r$ of the order of $S$ with the following properties:
	\begin{enumerate}
		\item $H \cap S < S;$
		\item $HS=\aut S;$
		\item for every $h
		\in H$ we can find an element $s \in S\cap H$ such that 
		$|h|_r \neq |h s|_r$.
				\end{enumerate}
\end{lemma}
\begin{proof}
First suppose that $S$ is an alternating group of degree $n$ (with $n\neq 6$) or a sporadic simple group. We claim that in this case we can tale $r=2$ and $H \in \syl_2(\aut S).$ Indeed $|\aut S:S|\leq 2$ and if $\aut(S)	\neq S$, then there exists an involution $a \in \aut S \setminus S.$ To conclude it suffices to notice that $\aut  S$ contains an element $b$ of order 4 such that $\aut(S)=S\langle b \rangle$ (see for example \cite[Theorem 2]{akl}). Now assume that $S$ is a simple group of Lie type, defined over a field of characteristic $p$. The proof of the Lemma in \cite[Section 2]{uni} implies that, except when $S=\psl(2,q)$ and $q$ is odd, we can take $r=p$ and $H=N_G(P)$ for $P \in \syl_p(S).$ We remain with the case when $S=\psl(2,q)$ and $q$ is odd. In this case let $r=2,$ $P\in \syl_2(G)$ and $H=N_G(P).$
We may choose $P$ so that the Frobenius automorphism $\sigma$ belongs to $H.$
Let $h \in H.$ Up to multiplying with a suitable element of $S\cap H$, we may assume $h=y\sigma_1\sigma_2$, with $y \in H$, $|y|=2$,
$\sigma_1,\sigma_2 \in \langle \sigma \rangle,$ $|\sigma_1|$ a 2-power, $|\sigma_2|$ odd and $[y,\sigma]=1.$ Let $\tilde q$ be the size of the subfield of $GF(q)$ centralized by $\sigma_2.$
By \cite[Theorem 2]{akl}, there exists $t \in \psl(2,\tilde q)\langle h\sigma_1\rangle \cap H$ such that $|ty\sigma_1|_2>|y\sigma_1|_2.$ We have $|th|_2>|h|_2.$
\end{proof}
\begin{lemma}\label{spre}
	Suppose that $h_1, h_2\in H.$ There exist $s \in S,$ $t \in H\cap S$ such that $\langle h_1s, h_2t\rangle = \langle h_1, h_2\rangle S.$
\end{lemma}
\begin{proof}
It follows from Lemma \ref{accar} (3) that there exists $t \in H \cap S$ such that $\langle h_2t\rangle \cap S \neq 1.$ Indeed if $\langle h_2\rangle \cap S \neq 1,$ then we can take $t=1.$ Otherwise take $t\in H \cap S$ such that $|h_2|_r \neq |h_2t|_r;$ then $\langle h_2t \rangle \cap S$ contains a non-trivial $r$-element.  Now let $1\neq u \in \langle h_2t \rangle \cap S.$
 By \cite[Theorem 1]{bgh}, there exists $s \in S$ such that $\langle h_1s, u\rangle=\langle h_1, S\rangle.$ This implies 
 $\langle h_1s, h_2t\rangle= \langle h_1, h_2\rangle S.$
\end{proof}
In the next two lemmas let $\tilde H=\{(\alpha_1,\dots,\alpha_n)\sigma \in G \mid \alpha_1,\dots,\alpha_n \in H\},$ being $H$ the subgroup of $\aut S$ introduced in the statement of Lemma \ref{accar}. Clearly, since $H\cap S < S,$ $\tilde H$ is a proper subgroup of $G.$

\begin{lemma}\label{brutto}
	Suppose $\langle g_1, g_2 \rangle A=G$ and that one of the following holds:
	\begin{enumerate}
		\item $g_1^\pi$ has a fixed point;
		\item $g_2^\pi$ has a fixed point;
	\item $(g_1g_2^i)^\pi$ is fixed-point-free for every $i \in \mathbb Z.$
	\end{enumerate}
Then there exist $u_1, u_2$ in $A$ such that $\langle u_1g_1, u_2g_2 \rangle=G$ and $u_1g_1 \in \tilde H$.
\end{lemma}
\begin{proof}
We will set $\epsilon=2$ if $g_2^\pi$ has a fixed point and $g_1^\pi$ is fixed-point-free, $\epsilon=1$ otherwise. Moreover let
$(\xi_1,\xi_2)=(g_1,g_2)$ if $\epsilon=1,$ $(\xi_1,\xi_2)=(g_2,g_1)$
otherwise. Since $SH=\aut S,$ there exists
$\bar \xi_1, \bar \xi_2 \in \tilde H$ with $\bar \xi_1A=\xi_1$ and $\bar \xi_2A=\xi_2.$ Clearly $\langle \bar \xi_1, \bar \xi_2\rangle A=G.$ The proof of \cite[Theorem 1.1]{uni} indicates how to select $u_1, u_2 \in A$ such that $\langle u_1\bar \xi_1, u_2\bar \xi_2\rangle=G.$ 
We check carefully the proof of that theorem in order to see that $u_1, u_2$ can be chosen with the additional property that $u_\epsilon \in A \cap H^n$.

 Let $r$ be the prime appearing in the statement of Lemma \ref{accar}. A quasi-ordering relation on the set
	of the cyclic permutations which belong to the group $\sym(n)$ is defined. 
	Let $\sigma _1$, $\sigma _2 \in
	\sym(n)$ be two cyclic permutations (including cycles of
	length 1); we define $\sigma _1 \leq \sigma _2$ if either
	$|\sigma _1|_r <
	|\sigma _2|_r$  or
	$|\sigma _1|_r =|\sigma _2|_r$ and $|\sigma _1| \leq |\sigma
	_2|$.
	
	Let $\bar \xi_1=(\alpha_1,\dots,\alpha_n)\rho$,
	$\bar \xi_2=(\beta_1,\dots,\beta_n)\sigma$, with $\alpha_i, \beta_j \in \Aut S$
	and $\rho, \sigma \in \sym(n)$.
	Then write $\rho=\rho_1 \cdots \rho_{s(\rho)}$,
	$\sigma=\sigma_1 \cdots \sigma_q \cdots \sigma_{s(\sigma)}$
	as product of disjoint cycles (including possibly cycles of length 1),
	in such a way that:
	\begin{itemize}
		\item[a)] $\rho_1 \leq \dots \leq \rho_{s(\rho)}$;
		\item[b)] $\supp (\sigma _i) \cap \supp (\rho_1) \neq \emptyset$ if and only if
		$i \leq q$;
		\item[c)] $\sigma_1 \leq \dots \leq \sigma_q$.
	\end{itemize}
	
	Let $\rho_i=(m_{i,1},\dots,m_{i,|\rho_i|})$, $1 \leq i \leq s(\rho)$,
	$\sigma_j=(n_{j,1},\dots,n_{j,|\sigma_j|})$, $1 \leq j \leq q$. We assume
	$m_{1,1}=n_{1,1}=m$. 
If $u_1=(x_1,\dots,x_n)$, $u_2=(y_1,\dots,y_n) \in A=S^n$ define
	$$\bar \alpha_r=x_r\alpha_r,\quad \bar \beta_r=y_r\beta_r,\quad 1 \leq r \leq n,$$
	$$\begin{aligned}a_i =& \ \bar\alpha_{m_{i,1}} \cdots \bar\alpha_{m_{i,|\rho_i|}},
	\ 1 \leq i \leq s(\rho),\\
	b_j =&\ \bar\beta_{n_{j,1}} \cdots \bar\beta_{n_{j,|\sigma_j|}},
	\ 1 \leq j \leq q.\end{aligned}$$
		Moreover let
	$a =\alpha_{m_{1,1}} \bar\alpha_{m_{1,2}} \cdots \bar\alpha_{m_{1,|\rho_1|}}$,
	$b =\beta_{n_{1,1}} \bar\beta_{n_{1,2}} \cdots \bar\beta_{n_{1,|\sigma_1|}}$ and consider
	$K=\langle a, b, S \rangle$. 
Now we say  that a $2n$-tuple $(x_1,\dots,x_n,y_1,\dots,y_n)\in S^{2n}$ is {\em good} if the following three conditions are satisfied:
	
	\begin{itemize}
		\item[(1)] $\langle a_1, b_1 \rangle=
		\langle x_ma, y_mb \rangle=K$.
		\item[(2)] If $2 \leq i \leq s(\rho)$, then $a_i^{|\rho_1 \cdots \rho_i|/|\rho_i|}$ is not conjugate to
		$a_1^{|\rho_1 \cdots \rho_i|/|\rho_1|}$  in $\Aut S.$
		\item[(3)] If $2 \le j \le q$, then $b_j^{|\sigma_1 \cdots \sigma_j|/|\sigma_j|}$ is not conjugate to
		$b_1^{|\sigma_1 \cdots \sigma_j|/|\sigma_1|}$ in $\Aut S$.
	\end{itemize}

We claim that we can find a good $2n$-uple with the additional property that $\{x_1,\dots,x_n,y_1,\dots,y_n\}\setminus \{y_m\}\subseteq H.$ To construct this $2n$-uple we start by choosing arbitrarily $(x_1,\ldots,x_n,y_1,\ldots,y_n)\in (S\cap H)^{2n}$, and then we modify the elements $x_m, y_m,$ the elements
$x_i$ for 	$i \in \{m_{i,1}\mid 2 \leq i \leq s(\rho)\}$ and the elements $y_j$ for
$j \in \{n_{j,1}\mid 2 \leq j \leq q\},$ so that  (1), (2) and (3) are satisfied.
First by  Lemma \ref{spre}, we can find $x \in H\cap S$ and $y \in S$   such that $\langle x a, yb \rangle=K$. We substitute
the original $x_m, y_m$ with $xx_m$ and $yy_m.$
Then let  $i\in \{2,\dots,s(\rho)\}.$
Since $\rho_1 \le \rho_i$, 
 $|\rho_1 \cdots \rho_i|/|\rho_i|$ is coprime with $r$ and therefore,  
	by Lemma \ref{accar}, there exists 
	$s_i \in S\cap H$ such that $
	(s_ia_i)^{|\rho_1 \cdots \rho_i|/|\rho_i|}$ is not conjugate to 
	$a_1^{|\rho_1 \cdots \rho_i|/|\rho_1|}$ in $\aut S$.
	We substitute  $x_{m_{i,1}}$  with $s_i x_{m_{i,1}}.$
Finally let  $j\in \{2,\dots,q\}.$
Since $\sigma_1 \le \sigma_j$, 
$|\sigma_1 \cdots \sigma_j|/|\sigma_j|$ is coprime with $r$, 
again by Lemma \ref{accar}, there exists 
$t_j \in S\cap H$ such that $
(t_jb_j)^{|\sigma_1 \cdots \sigma_j|/|\sigma_j|}$ is not conjugate to 
$b_1^{|\sigma_1 \cdots \sigma_j|/|\sigma_1|}$ in $\aut S$.
We substitute  
$y_{n_{j,1}}$ with $t_jy_{n_{j,1}}.$
With the same argument we can prove that there exists a good $2n$-uple with the additional property that $\{x_1,\dots,x_n,y_1,\dots,y_n\}\setminus \{x_m\}\subseteq H.$	
	
If follows from the proof of \cite [Theorem 1.1]{uni} that if either $\rho$ or $\sigma$ has a fixed point and
$(x_1,\dots,x_n,y_1,\dots,y_n)$ is a good $2n$-tuple then $G=\langle u_1\bar \xi_1,
u_2\bar \xi_2\rangle,$ with $u_1=(x_1,\dots,x_n)$ and $u_2=(y_1,\dots,y_n).$ So it follows from the previous paragraph that $u_1,u_2$ can be chosen so that $u_\epsilon \in (S\cap H)^n.$
	
We remain with the case when $\rho\sigma^i$ is fixed-point-free for every $i \in \mathbb Z.$ In this case
by our definition $\epsilon=1.$ Choose
$u_1=(x_1,\dots,x_n)$, $u_2=(y_1,\dots,y_n) \in S^n$ with the properties that $(x_1,\dots,x_n,y_1,\dots,y_n)$ is a good $2n$-tuple with $\{x_1,\dots,x_n\}\in H\cap S.$ This is the case for which the proof of \cite[Theorem 1.1]{uni} requires more work, since the condition that $(x_1,\dots,x_n,y_1,\dots,y_n)$ is a good $2n$-tuple is no more sufficient to ensure that $G=\langle u_1\bar \xi_1,
u_2\bar \xi_2\rangle.$ Additional conditions are required, but these conditions can be satisfied with further modifications involving only the elements $y_k$ for $k \notin \{m\} \cup\{n_{j,1}\mid 2 \leq j \leq q\}$. In particular it is not needed to modify any more the elements $x_1,\dots,x_n$. This means that there exists $\bar u_2 \in A$ such that
$G=\langle u_1\bar \xi_1,
\bar u_2\bar \xi_2\rangle$	and this is enough for our purpose.
\end{proof}

\begin{cor}\label{tildeg}
	Let $g$ be a vertex of $\Sigma(G).$ If $G/\soc(G)$ is not cyclic, then there exists a vertex $\tilde g \in \Sigma(G)$ such that $\tilde g \in \tilde H$ and the distance between $g$ and $\tilde g$ in $\Sigma(G)$ is at most 2. 
\end{cor}
\begin{proof}
Since $g \in \Sigma(G),$ there exists $g_2 \in G$ such that $\langle g, g_2\rangle =G.$ 
If $g$ and $g_2$ satisfy one of the three conditions in the statement of Lemma \ref{brutto}, then there exists $u, u_2$ in $\soc(G)=A$ such that
$\langle ug, u_2g_2 \rangle =G$ and $\tilde g= ug \in \tilde H.$
In particular $\tilde g \in \Sigma(G)$ and $\langle g, \tilde g \rangle \leq \langle g \rangle A < G,$ since we are assuming that $G/A$ is not cyclic, and therefore $g$ and $\tilde g$ are adjacent vertices of $\Sigma(G).$ 
Now assume that $g^\pi$ and $g_2^\pi$ are fixed-point-free but there exists $i\in \mathbb Z$ such that $g_3=gg_2^i$ has a fixed point. If $\langle g, g_3 \rangle=G,$ then we repeat the previous argument using $g_3$ instead of $g_2$ and the find an element $\tilde g$ with the required properties. Suppose $\langle g,g_3 \rangle \neq G.$ In any case $\langle g_2, g_3\rangle= \langle g_2, g\rangle=G$, so $g_3$ is a vertex of $\Sigma(G)$ which is adjacent to $g.$ Moreover, we may apply Lemma \ref{brutto} to the generating pair $(g_3,g_2)$ in order to find $\tilde g_3 \in 
\tilde H$ which is adjacent in $\Sigma(G)$ to $g_3$. So $(g, g_3,  \tilde g_3)$ is a path in $\Sigma(G)$
and  we may take $\tilde g=\tilde g_3.$
\end{proof}

\begin{prop}\label{unso}The graph $\Sigma(G)$ is connected, with diameter at most 5.
\end{prop}

\begin{proof}
	We distinguish two cases:

\noindent a) $G/\soc(G)$ is not cyclic.
Suppose that $g_1, g_2$ are two different vertices of $\Sigma(G).$
Choose $\tilde g_1, \tilde g_2$ as in the statement of Lemma \ref{tildeg}. Since $\langle \tilde g_1, \tilde g_2 \rangle \leq \tilde H < G,$ $\tilde g_1, \tilde g_2$ are adjacent vertices of $G$. So the distance in $\Sigma(G)$ between $g_1$ and $g_2$ is at most 5.

\noindent b) $G/\soc(G)$ is cyclic.
Recall that the intersection graph  $\I(G)$ of $G$ is the graph whose vertices are the non-trivial proper subgroups of $G$ and in which two vertices $H$ and $K$ are adjacent if and only if $H\cap K\neq 1.$ 
 By \cite[Theorem 1]{bgh}, when $G/\soc(G)$ is cyclic, the vertex set $V(G)$ of $\Sigma(G)$ coincides with the set of the non-trivial elements of $G$. In particular, as it is explained in \cite[Section 12]{cam},
($\Sigma(G), \I(G)$)  is a dual pair of graphs, and therefore there is a natural
bijection between connected components of $\Sigma(G)$ and connected components of
$\I(G)$ with the property that corresponding components have diameters which are
either equal or differ by 1.
If $G$ is neither soluble nor simple, then 
$\I(G)$ is connected with $\diam(\I(G))\leq 4$ (see \cite[Lemma 5]{she}) hence $\diam(\Sigma(G))\leq 5$.
Freedman recently proved that also when $G$ is a finite non-abelian simple group
the graph $\Sigma(G)$ is connected
with diameter  is at most 5 (see the remark after \cite[Proposition 12]{cam}).
\end{proof}

We don't know whether the bound in the previous proposition is the best possible.
However Freedman \cite{free} proved that $\diam(\I(G))\leq 5$ for any finite non-abelian simple group $G$ and that the upper bound is attained only by the Baby Monster $\B$ and
some unitary groups. In particular, since $(\I(\B),\Sigma(\B))$ is a dual pair of graphs, it follows $\diam(\Sigma(\B))\geq 4.$

\section{Groups with nilpotent derived subgroup}\label{ultimasec}

\begin{lemma}\label{ciclo}
If $G$ is a non-trivial cyclic group, then $\Sigma(G)$ is disconnected.
\end{lemma}
\begin{proof}
If $\langle g \rangle=G,$ then $g$ is an isolated vertex of $\Sigma(G).$
\end{proof}

\begin{lemma}\label{pgruppo}
If $G$ is a 2-generated finite $p$-group, then $\Sigma(G)$ is disconnected.
\end{lemma}
\begin{proof}
By the previous lemma, we may assume that $G$ is not cyclic. Let $F=\frat(G)$. Then $G$ has precisely $p+1$ maximal subgroups,
$M_1,\dots,M_{p+1}$ and $V(G)=G\setminus F.$
Moreover two distinct vertices $x$ and $y$ of $\Sigma(G)$ are adjacent if and only if
$x, y \in M_i$ for some $1\leq i \leq p+1.$
This implies that $\Sigma(G)$ is a complete $(p+1)$-mutipartite graph, with parts $M_1\setminus F,\dots,M_{p+1}\setminus F.$ 
\end{proof}

A crucial role in our proof will be played by the following result, due to Gasch\"utz.
\begin{prop}\cite{Ga}\label{gaz}
	Let $N$ be a normal
	subgroup of a finite group $G$ and suppose that $\langle g_1,\dots g_k\rangle N=G$. If $k \geq d(G),$ then there
	exist $n_1,\dots,n_k \in N$ so that $\langle g_1n_1,\dots g_kn_k\rangle=G$.
\end{prop}

\begin{prop}\label{quoz}
Let $N$ be a proper normal subgroup of $G.$ If $\Sigma(G/N)$ is connected, then $\Sigma(G)$ is connected and $\diam(\Sigma(G))\leq \diam(\Sigma(G/N)).$
\end{prop}
\begin{proof}
 Let $g_1,g_2$ be two different vertices of $\Sigma(G).$ If $g_1N=g_2N,$ then
$\langle g_1, g_2 \rangle \leq \langle g_1, g_2\rangle N < G,$ since, by Lemma \ref{ciclo}, $G/N$ is not cyclic. In this case $g_1,g_2$ are adjacent vertices of $\Sigma(G).$ So we may assume $g_1N\neq g_2N.$ In this case there exists a path $(g_1N,y_1N,\dots,y_rN,g_2N)$ in $\Sigma(G/N).$ By Proposition \ref{gaz}, for any $1\leq i\leq r$, $y_in_i \in V(G)$ for some $n_i \in N.$ Thus $(g_1,y_1n_1,\dots,y_rn_r,g_2)$ is a path in $\Sigma(G).$
\end{proof}

\begin{cor} If $G$ is non-soluble, then
	$\Sigma(G)$ is connected and $\diam(\Sigma(G))\leq 5.$
\end{cor}
\begin{proof}
A non-soluble group $G$ has an epimorphic image which is monolithic with non-abelian socle. So the conclusion follows combining the previous lemma with  Proposition \ref{unso}.
\end{proof}

A chief factor $A=X/Y$ of a finite group $G$ is said to be a non-Frattini chief factor if $X/Y \not\leq \frat(G/Y).$

\begin{cor}Let $G$ be a 2-generated soluble group.
If there exists a non-Frattini and non-central chief factor $A$ of $G$ such that $G/C_G(A)$ is not cyclic of prime order, then $G$ is connected and $\diam(\Sigma(G))\leq 3.$
\end{cor}

\begin{proof}
If $A$ is a non-Frattini chief factor of $G$, then $G$ admits as an epimorphic image
the semidirect product $A\rtimes G/C_G(A),$ so  the conclusion follows combining Lemma \ref{quoz} with  Proposition \ref{monosol}.
\end{proof}

\begin{lemma}\label{pochi} Let $G$ be a 2-generated soluble group. Then one of the following occurs:
	\begin{enumerate}
		\item $\Sigma(G)$ is connected and $\diam(\Sigma(G))\leq 3;$
		\item The derived subgroup of $G$ is nilpotent and 
		$G/\frat(G)$ has the following structure:
		$$G/\frat(G)\cong (V_1\times \dots \times V_t)\rtimes H,$$
		where $H$ is abelian and $V_1,\dots,V_t$
		are pairwise non $H$-isomorphic non-trivial irreducible $H$-module (including the possibility $t=0$).
	\end{enumerate}
\end{lemma}
\begin{proof}
Let $\mathcal A$ be the set of the non-trivial irreducible $G$-modules that are $G$-iso\-morphic to a non-Frattini chief factor. By the previous corollary we may assume that $G/C_G(A)$ is cyclic of prime order for every $A\in \mathcal A.$ Let $F$ be the Frattini subgroup of $G$. Then the Fitting subgroup $W/F$ of $G/F$ has a complement in $G/F$ and it is a direct product of minimal normal subgroups of $G/F.$ In particular $G/F \cong (V_1\times \dots \times V_u) \rtimes K,$  where $V_j$ is an irreducible $K$-module for $1\leq j \leq u$ and $\bigcap_{1\leq j \leq u}C_K(V_j)=1.$ By assumption, for $1\leq j\leq u,$ either $C_K(V_j)=K$ or $K/C_K(V_j)\cong C_{p_j}$ for some prime $p_j.$  In particular
$K \leq \prod_{1\leq j\leq u}K/C_K(V_j)$ is abelian.
We may assume that $V_j \leq Z(G)$ if and only if $j>t.$ So we have
$G/F\cong (V_1\times \dots \times V_t) \rtimes H$, with $H=(V_{t+1}\times \dots \times V_u)\times K.$ If $1\leq j_1 < j_2\leq t,$ then $(V_{j_1}\times V_{j_2})\rtimes H,$ being an epimorphic image of $G$, must be 2-generated, and this implies
$V_{j_1}\not\cong_H V_{j_2}.$
\end{proof}

For the remaining part of this section, we concentrate our attention on the finite groups $G$ with nilpotent derived subgroup.
First we consider the particular case when $G$ itself is nilpotent. The analysis of this case relies on the following lemma.
\begin{lemma}\label{prod}
	If $G=A \times B$ is a non-cyclic group and $(|A|,|B|)=1,$ then
	$\Sigma(G)$ is connected and $\diam(\Sigma(G))\leq 3,$
	with equality if and only one of the following occurs:
	\begin{enumerate}
		\item $A$ is cyclic and either $\Sigma(B)$ is disconnected or
		 $\diam(\Sigma(B))>2.$
		\item $B$ is cyclic and either $\Sigma(A) $ is disconnected or $\diam(\Sigma(A))>2.$
	\end{enumerate}
	
\end{lemma}

\begin{proof} Since $(|A|,|B|)=1,$ two elements $(x_1,y_1)$, $(x_2,y_2)$ generate $G$ if and only if $\langle x_1, x_2\rangle=A$
	and $\langle y_1, y_2\rangle=B$. In particular $V(G)=V(A)\times V(B).$ Since $G$ is not cyclic and the orders of $A$ and $B$ are coprime, it is not restrictive to assume that $A$ is also non-cyclic.
	Suppose that $g_1=(a_1,b_1)$, $g_2=(a_2,b_2)$ are two different vertices of $\Sigma(G).$ 
	If $B\neq \langle b_1 \rangle$, then $((a_1,b_1),(a_2,b_1),(a_2,b_2))$ is a path in $\Sigma(G).$ Similarly,
if	$B\neq \langle b_2 \rangle$, then $((a_1,b_1),(a_1,b_2),(a_2,b_2))$ is a path in $\Sigma(G).$
	If $B= \langle b_1 \rangle=\langle b_2 \rangle$, then $((a_1,b_1),(a_1,1),$ $(a_2,1),(a_2,b_2))$ is a path in $\Sigma(G)$. In the last case, if $(a_1,a,a_2)$ is a path is $\Sigma(A),$ then $((a_1,b_1), (a,b_1),
	(a_2,b_2))$ is also a path is $\Sigma(G).$

Finally assume that  $B=\langle b \rangle$ is cyclic and there exist $a_1, a_2 \in V(A)$ without a common neighbour in $\Sigma(A).$  Then $(a_1,b)$ and $(a_2,b)$ do not have a common neighbour in $\Sigma(G)$ and therefore $\diam(\Sigma(G))\geq 3.$
\end{proof}

\begin{cor}\label{nilpo}
	Let $G$ be a 2-generated finite nilpotent group, Then $\Sigma(G)$ is disconnected if and only if $G$ is either a cyclic group or a $p$-group.
	Moreover if $G$ is neither a cyclic group nor a $p$-group, then $\diam(\Sigma(G))=3$ if $G$ has only one non-cyclic Sylow subgroup, $\diam(\Sigma(G))=2$ otherwise.
\end{cor}
\begin{proof}
We decompose $G=P_1\times \cdots \times P_t\times Q_1\times \cdots \times Q_u$
where $P_1,\dots,P_t$ are the cyclic Sylow subgroups of $G$ and
$Q_1,\dots,Q_u$ are the remaining Sylow subgroups. If $u=0$, then 
$\Sigma(G)$ is disconnected by Lemmas \ref{ciclo}. So we may assume $u\geq 1$. If $u=1$ and $t=0,$ then $\Sigma(G)$ is disconnected by Lemma \ref{pgruppo}. If $u=1$ and $t>0,$ then $G=X \times Q_1$
with $X=P_1\times \dots \times P_t$. So by Lemma \ref{prod}, $\Sigma(G)$ is connected,
and, since $X$ is cyclic and $\Sigma(Q_1)$ is disconnected,  
 $\diam(\Sigma(G))=3.$
If $u\geq 2,$ then $G=K \times Q_t$ with
$K=P_1\times \cdots \times P_t\times Q_1\times \cdots \times Q_{u-1}.$ Neither $K$ nor $Q_u$ is cyclic, so, again by Lemma \ref{prod}, $\Sigma(G)$ is connected and $\diam(\Sigma(G))=2.$
\end{proof}

It remains to investigate case when $G$ is as in Lemma \ref{pochi} (2) and $t>0.$ By the following Lemma it is not restrictive to assume $\frat(G)=1.$
\begin{lemma}
	Let $G$ be a 2-generated finite group. Then $\Sigma(G)$ is connected if and only if $\Sigma(G/\frat(G))$ is connected.
\end{lemma}

\begin{proof}Let $F=\frat(G).$ Since $G$ is cyclic if and only if $G/F$ is cyclic, by Lemma 	\ref{ciclo} we may assume that $G$ is not cyclic. By Proposition \ref{quoz} we have only to prove that if $\Sigma(G/F)$ is disconnected, then $\Sigma(G)$ is also disconnected. So
	assume that $\Omega$ is a connected component of $\Sigma(G/F)$ and that there exists $yF \in V(G/F) \setminus \Omega.$ Since $\langle g_1, g_2 \rangle=G$ if and only if $\langle g_1F, g_2F \rangle=G/F,$ it follows immediately that $\Omega^*=\{x 	\in G\mid xF \in \Omega\}$ is a connected component of $\Sigma(G)$ and $y \in V(G) \setminus \Omega^*.$
\end{proof}

\begin{lemma}\label{restri}Let $G\cong (V_1\times \dots \times V_t)\rtimes H$,
		where $H$ is abelian and $V_1,\dots,V_t$
	are pairwise non $H$-isomorphic non-trivial irreducible $H$-modules. Then 
one of the following occurs:
\begin{enumerate}
	\item $\Sigma(G)$ is connected and $\diam(G)\leq 3;$
	\item $H \cong C_p$ for a suitable prime $p$.
	\item $H\cong C_p^2$ for a suitable prime $p.$
\end{enumerate}
\end{lemma}
\begin{proof}
Let $C_i=C_H(V_i)$ and $H_i=H/C_i.$ The primitive soluble group $K_i=V_i \rtimes H_i$ is an epimorphic image of $G$. If $|H_i|$ is not a prime, then by Propositions \ref{monosol} and \ref{quoz}, $\Sigma(G)$ is connected and $\diam(\Sigma(G)\leq 3.$ So we may assume $|H_i|=p_i,$ with $p_i$ a prime for $1\leq i\leq t.$ Suppose $p_i\neq p_j$ for some $1\leq i<j\leq t.$ Then $K_i \times K_j$ is an epimorphic image of $G$ and it can be easily seen that $V(K_1 \times K_2)=V(K_1)\times V(K_2)$. Arguing as in the proof of Lemma \ref{prod}, it can be deduced that $\Sigma(K_1\times K_2),$ and consequently $\Sigma(G),$ is connected with diameter at most 3. So $p_i=p$ for
$1\leq i \leq t.$ Moreover $\bigcap_{1\leq i\leq t}C_i=1$, and therefore $H$ is an abelian group of exponent $p.$ Since $G$ is 2-generated, we conclude that $H\cong C_p^d$ with $d\leq 2.$
\end{proof}

In  the situation of the previous Lemma,  let $F_i=\kndo_H(V_i).$ Since $H$ is abelian, $\dim_{F_i}V_i=1$.
We may identify $V_i$ with the additive group of the field $F_i.$ Moreover, if $h \in H,$ then there exists $\alpha_i(h) \in F_i^*$ such that
$v^h=\alpha_i(h)v$ for every $v\in V_i.$ The following holds: 
\begin{lemma}\label{corone}
	Assume that $g_1=(v_{1,1},\dots,v_{1,t})h_1$ and $g_2=(v_{2,1},\dots,v_{2,t})h_2$ are elements of $G.$
For $1\leq j\leq t$ consider the matrix
	$$A_j:=\begin{pmatrix}
	1-\alpha_j(h_1)&1-\alpha_j(h_2)\\
	v_{1,j}&v_{2,j}
	\end{pmatrix}.
	$$
	Then $\langle g_1,g_2\rangle=G$ if and only if the following holds:
	\begin{enumerate}
		\item  $\langle h_1,h_2\rangle=H;$
		\item $\det(A_j)\neq 0$ for every $1\leq j\leq t.$
	\end{enumerate}
\end{lemma}
\begin{proof}
	See Proposition 2.1 and Proposition 2.2 in \cite{cliq}.
\end{proof}

\begin{lemma}
	Let $G$ be as in Lemma \ref{restri}. If $H \cong C_p,$ then $\Sigma(G)$ is disconnected.
\end{lemma}

\begin{proof}
	Let $W=V_1\times \dots \times V_t.$ By Lemma \ref{corone}, if $w=(v_1,\dots,v_t) \in W$, then $w \in V(G)$ if and only if $v_i \neq 0$ for $1\leq i \leq t.$ Consider $\Omega=W \cap V(G).$ If $w_1, w_2$ are two different elements of $\Omega,$ then they are adjacent in $\Sigma(G).$ Again by Lemma $\ref{corone},$ if $g\in V(G)\setminus \Omega,$ then $G=\langle g, w\rangle$ for any $w \in \Omega.$ This implies that
	$\Omega$ is a proper connected component of $\Sigma(G).$
\end{proof}

\begin{lemma}
	Let $G$ be as in Lemma \ref{restri}. If $H \cong C_p\times C_p$ and $Z(G)\neq 1,$ then $\Sigma(G)$ is disconnected.
\end{lemma}

\begin{proof}
In this case $Z(G)=\langle h\rangle$ is a subgroup of $G$ of order $p.$	Let $W=V_1\times \dots \times V_t.$ By Lemma \ref{corone}, if $x=(v_1,\dots,v_t)h^j \in W\langle h\rangle$, then $x \in V(G)$ if and only if $h^j\neq 1$ and $v_i \neq 0$ for $1\leq i \leq t.$ Consider $\Omega=W\langle h\rangle \cap V(G).$ If $x_1, x_2$ are two different elements of $\Omega,$ then they are adjacent in $\Sigma(G).$ Again by Lemma $\ref{corone},$ if $g\in V(G) \setminus \Omega,$ then $G=\langle g, x\rangle$ for any $x \in \Omega.$ This implies that
	$\Omega$ is a proper connected component of $\Sigma(G).$
\end{proof}

\begin{lemma}
	Let $G$ be as in Lemma \ref{restri}. If $H \cong C_p\times C_p$ and $Z(G)= 1,$ then $\Sigma(G)$ is connected
	and $\diam(\Sigma(G))\leq 2.$
\end{lemma}
\begin{proof}
	For $h \in H,$ let $\Delta(h)=\{i \in \{1,\dots,t\}\mid h \in C_H(V_i)\}.$ 
	 Let $g=(v_1,\dots,v_t)h \in G.$ By Lemma \ref{corone}, $g\in V(G)$ if and only if $h\neq 1$ and $v_j\neq 0$ for any $j\in \Delta(h).$ Suppose that $g_1=(x_1,\dots,x_t)h_1$, $g_2=(y_1,\dots,y_t)h_2$ are two distinct vertices of $\Sigma(G).$ We may assume $\langle g_1, g_2\rangle=G$, otherwise $g_1,g_2$ are adjacent vertices of $\Sigma(G).$ Up to reordering, we may assume $\Delta(h_1)=\{1,\dots,r\}$ for some $r \in \{0,\dots,t\}.$ Since $H=\langle h_1, h_2\rangle$ and $Z(G)=\bigcap_{1\leq j \leq t}C_H(V_j)=1,$ we must have
	$\Delta(h_2)\subseteq \{r+1,\dots,t\}.$ Up to reordering we may assume $\Delta(h_2) = \{r+1,\dots,r+s\}$ for some $s \in \{0,\dots,t-r\}.$ Moreover, up to conjugation with a suitable element of $V_1\times \dots \times V_t,$ we may assume $x_j=0$ if $j > r$ and $y_k=0$ if $k\leq r.$
	If $s< t,$ then  $g=(0,\dots,0,y_{r+1},\dots,y_{r+s},0\dots,0)h_2\in V(G).$
	On the other hand $\langle g_1, g\rangle$ is
	 contained in $(V_1\times \dots \times V_{t-1})\rtimes H$ and $\langle g_2, g\rangle$ is contained in $(V_1\times \dots \times V_{t})\rtimes \langle h_2 \rangle,$
	 	 so $(g_1,g,g_2)$ is a path in $\Sigma(G).$ Finally assume $r+s=t.$ In this case $\Delta(h_1h_2)=\emptyset$, so $h_1h_2\in V(G)$. Moreover $r>0,$ otherwise $\Delta(h_2)=\{1,\dots,t\}$ and $h_2 \in Z(G),$
	and $r<t$ otherwise $\Delta(h_1)=\{1,\dots,t\}$ and $h_1 \in Z(G).$
	Thus $\langle g_1, h_1h_2\rangle\leq (V_1\times \dots \times V_{t-1})\rtimes H$ and  $\langle g_2, h_1h_2\rangle \leq (V_2\times \dots \times V_{t})\rtimes H$. But then
$(g_1,h_1h_2,g_2)$ is a path in $\Sigma(G).$ 
\end{proof}

\begin{proof}[Proof of Proposition \ref{isola}]First notice that
 generators of a cyclic group and  involutions in $C_2\times C_2$ and $D_p$
 are isolated vertices in the corresponding graphs. Conversely, let $G$ be a 2-generated finite group and suppose that $g\in G$ is an isolated vertex of $\Sigma(G)$. If $g\neq g^{-1}$, then $G=\gen{g,g^{-1}}=\gen{g}$, so $G$ is cyclic. Otherwise $g$ is an involution.
 Suppose that this is the case and assume  that $G$ is not cyclic.
 Since $g \in V(G),$ the set $V_g=\{h \in G\mid \langle h, g\rangle=G\}$ is non-empty. Suppose $h \in V_g.$ Then $G=\langle g,h\rangle = \langle gh^i, h\rangle$ for any $i \in \mathbb Z.$ Hence $gh^i \in V(G)$ and therefore either $h^i=1$ or $\langle g, gh^i\rangle=\langle g, h^i\rangle=G$. In other words, if $g$ generates $G$ together with $h$, then it generates $G$ together with any non-trivial power of $h.$ If $\langle g,h \rangle$ is abelian, this is possible only if $|h|$ is a prime, and we must have $|h|=2$ otherwise $G$ would be cyclic. So if $G$ is abelian, then $G\cong C_2\times C_2.$
If $G$ is non-abelian and $h \in V(G)$, then $g\neq g^h \in V(G),$
thus  $G=\gen{g,g^h}\cong D_n$ is a dihedral group of order $2n$ with $n=|gg^h|.$ In particular $g$ generated $G$ together with any non-trivial power of $gg^h$ and this is possible only if $n$ is a prime.
\end{proof}

\end{document}